\documentclass[11pt,reqno]{amsart}
\usepackage{amsmath,amssymb,amsthm,color}
\usepackage[T1]{fontenc}
\usepackage{lmodern}
\usepackage{graphicx}
\usepackage{geometry}
\usepackage[latin1]{inputenc}
\usepackage{hyperref}
\usepackage{enumerate}
\newif\ifmicrotype
\relax
\AtBeginDocument{%
  \def\tempaa{\thesection.\arabic{equation}}%
  \ifx\theequation\tempaa
     \def\theequation{\thesection\hbox{--}\arabic{equation}}%
  \fi
  \ifmicrotype
     \RequirePackage{microtype}%
     \DeclareMicrotypeAlias{mnt}{ptm}%
  \fi
  \DeclareMathSizes{10}   {10}   {7}{6}% from msp.cls
  \DeclareMathSizes{10.95}{10.95}{8}{6}% from msp.cls
}
\DeclareFontFamily{OT1}{ptm}{}
\DeclareFontShape{OT1}{ptm}{m}{n} { <-> ptmr}{}
\DeclareFontShape{OT1}{ptm}{m}{it}{ <-> ptmri}{}
\DeclareFontShape{OT1}{ptm}{m}{sl}{ <->ptmro}{}
\DeclareFontShape{OT1}{ptm}{m}{sc}{ <-> ptmrc}{}
\DeclareFontShape{OT1}{ptm}{b}{n} { <-> ptmb}{}
\DeclareFontShape{OT1}{ptm}{b}{it}{ <-> ptmbi}{}     
\DeclareFontShape{OT1}{ptm}{bx}{n} {<->ssub * ptm/b/n}{}
\DeclareFontShape{OT1}{ptm}{bx}{it}{<->ssub * ptm/b/it}{}

\DeclareSymbolFont{bold}{OT1}{ptm}{b}{n}
\DeclareMathAlphabet{\mathbf}{OT1}{ptm}{b}{n}  
\DeclareMathAlphabet{\mathrm}{OT1}{ptm}{m}{n}

\DeclareFontFamily{OT1}{psy}{}      
\DeclareFontShape{OT1}{psy}{m}{n}{ <-> s * [0.9] psyr}{}
\DeclareFontFamily{OMS}{ptm}{}     
\DeclareFontShape{OMS}{ptm}{m}{n}{ <8> <9> <10> gen * cmsy }{}
\DeclareFontFamily{OMS}{cmtt}{}     
\DeclareFontShape{OMS}{cmtt}{m}{n}{ <8> <9> <10> gen * cmsy }{}

\SetSymbolFont{operators}{normal}{OT1}{ptm}{m}{n}   
\SetSymbolFont{operators}{bold}{OT1}{ptm}{b}{n}     
\DeclareSymbolFont{emsy}{OT1}{ptm}{m}{it}
\DeclareSymbolFont{emsr}{OT1}{ptm}{m}{n}
\DeclareSymbolFont{emcmr}{OT1}{cmr}{m}{n}   
\DeclareSymbolFont{emsymb}{OT1}{psy}{m}{n}  
\DeclareMathSymbol a{\mathalpha}{emsy}{"61}
\DeclareMathSymbol b{\mathalpha}{emsy}{"62}
\DeclareMathSymbol c{\mathalpha}{emsy}{"63}
\DeclareMathSymbol d{\mathalpha}{emsy}{"64}
\DeclareMathSymbol e{\mathalpha}{emsy}{"65}
\DeclareMathSymbol f{\mathalpha}{emsy}{"66}
\DeclareMathSymbol g{\mathalpha}{emsy}{"67}
\DeclareMathSymbol h{\mathalpha}{emsy}{"68}
\DeclareMathSymbol i{\mathalpha}{emsy}{"69}
\DeclareMathSymbol j{\mathalpha}{emsy}{"6A}
\DeclareMathSymbol k{\mathalpha}{emsy}{"6B}
\DeclareMathSymbol l{\mathalpha}{emsy}{"6C}
\DeclareMathSymbol m{\mathalpha}{emsy}{"6D}
\DeclareMathSymbol n{\mathalpha}{emsy}{"6E}
\DeclareMathSymbol o{\mathalpha}{emsy}{"6F}
\DeclareMathSymbol p{\mathalpha}{emsy}{"70}
\DeclareMathSymbol q{\mathalpha}{emsy}{"71}
\DeclareMathSymbol r{\mathalpha}{emsy}{"72}
\DeclareMathSymbol s{\mathalpha}{emsy}{"73}
\DeclareMathSymbol t{\mathalpha}{emsy}{"74}
\DeclareMathSymbol u{\mathalpha}{emsy}{"75}
\DeclareMathSymbol v{\mathalpha}{emsy}{"76}
\DeclareMathSymbol w{\mathalpha}{emsy}{"77}
\DeclareMathSymbol x{\mathalpha}{emsy}{"78}
\DeclareMathSymbol y{\mathalpha}{emsy}{"79}
\DeclareMathSymbol z{\mathalpha}{emsy}{"7A}
\DeclareMathSymbol A{\mathalpha}{emsy}{"41}
\DeclareMathSymbol B{\mathalpha}{emsy}{"42}
\DeclareMathSymbol C{\mathalpha}{emsy}{"43}
\DeclareMathSymbol D{\mathalpha}{emsy}{"44}
\DeclareMathSymbol E{\mathalpha}{emsy}{"45}
\DeclareMathSymbol F{\mathalpha}{emsy}{"46}
\DeclareMathSymbol G{\mathalpha}{emsy}{"47}
\DeclareMathSymbol H{\mathalpha}{emsy}{"48}
\DeclareMathSymbol I{\mathalpha}{emsy}{"49}
\DeclareMathSymbol J{\mathalpha}{emsy}{"4A}
\DeclareMathSymbol K{\mathalpha}{emsy}{"4B}
\DeclareMathSymbol L{\mathalpha}{emsy}{"4C}
\DeclareMathSymbol M{\mathalpha}{emsy}{"4D}
\DeclareMathSymbol N{\mathalpha}{emsy}{"4E}
\DeclareMathSymbol O{\mathalpha}{emsy}{"4F}
\DeclareMathSymbol P{\mathalpha}{emsy}{"50}
\DeclareMathSymbol Q{\mathalpha}{emsy}{"51}
\DeclareMathSymbol R{\mathalpha}{emsy}{"52}
\DeclareMathSymbol S{\mathalpha}{emsy}{"53}
\DeclareMathSymbol T{\mathalpha}{emsy}{"54}
\DeclareMathSymbol U{\mathalpha}{emsy}{"55}
\DeclareMathSymbol V{\mathalpha}{emsy}{"56}
\DeclareMathSymbol W{\mathalpha}{emsy}{"57}
\DeclareMathSymbol X{\mathalpha}{emsy}{"58}
\DeclareMathSymbol Y{\mathalpha}{emsy}{"59}
\DeclareMathSymbol Z{\mathalpha}{emsy}{"5A}
\DeclareMathSymbol{\bullet}{\mathalpha}{emsymb}{"B7}
\DeclareMathSymbol{\regis}{\mathalpha}{emsymb}{"D2}
\def\Bullet{\leavevmode\unkern{$\m@th\bullet$}\kern.32em\ignorespaces}
\def\Regis{\leavevmode\raise.5ex\hbox{$\m@th\regis$}}
\DeclareMathSymbol +{\mathbin}{emcmr}{`+}
\DeclareMathSymbol ={\mathrel}{emcmr}{`=}  
\DeclareMathSymbol{\Gamma}{\mathalpha}{emcmr}{"00}
\DeclareMathSymbol{\Delta}{\mathalpha}{emcmr}{"01}
\DeclareMathSymbol{\Theta}{\mathalpha}{emcmr}{"02}
\DeclareMathSymbol{\Lambda}{\mathalpha}{emcmr}{"03}
\DeclareMathSymbol{\Xi}{\mathalpha}{emcmr}{"04}
\DeclareMathSymbol{\Pi}{\mathalpha}{emcmr}{"05}
\DeclareMathSymbol{\Sigma}{\mathalpha}{emcmr}{"06}
\DeclareMathSymbol{\Upsilon}{\mathalpha}{emcmr}{"07}
\DeclareMathSymbol{\Phi}{\mathalpha}{emcmr}{"08}
\DeclareMathSymbol{\Psi}{\mathalpha}{emcmr}{"09}
\DeclareMathSymbol{\Omega}{\mathalpha}{emcmr}{"0A}
\DeclareMathSizes{7.6}{8}{6}{5}

\makeatletter
\newtheorem*{rep@theorem}{\rep@title}
\newcommand{\newreptheorem}[2]{%
\newenvironment{rep#1}[1]{%
\def\rep@title{#2 \ref{##1}}%
\begin{rep@theorem}}%
{\end{rep@theorem}}}
\makeatother

\newtheorem{theorem}{Theorem}[section]
\newreptheorem{theorem}{Theorem}
\newtheorem{lemma}[theorem]{Lemma}
\newtheorem{lem}[theorem]{Lemma}

\newtheorem{prop}[theorem]{Proposition}
\newtheorem{corollary}[theorem]{Corollary}

\theoremstyle{definition}
\newtheorem{definition}[theorem]{Definition}

\newtheorem{defi}[theorem]{Definition}
\newtheorem{remark}[theorem]{Remark}

\newtheorem{example}[theorem]{Example}

\newtheorem{observation}[theorem]{Observation}

\newtheorem{notation}[theorem]{Notation}

\def\é{\'e}
\def\è{\`e}
\def\ê{\^e}
\def\à{\`a}
\def\â{\^a}
\def\î{\^i}
\def\ô{\^o}
\def\ù{\`u}
\def\û{\^u}
\def\É{\'E}
\def\È{\`E}
\def\Ê{\^E}
\def\À{\`A}
\def\Ô{\^O}
\def\ç{\c{c}}

\def\Z{{\mathbb{Z}}}
\def\N{{\mathbb{N}}}

\def\co{{\colon \thinspace}}

\def\<{\langle}
\def\>{\rangle}

\def\-{\setminus}
\def\inv{^{-1}}
\renewcommand{\phi}{{\varphi}}
\newcommand\B{\mathcal B}

\newcommand\nintr[1]{\rho_n^{#1}}
\newcommand\muR{\mu_R}
\newcommand\muC{\mu_M} % changed
\newcommand\mupA{\mu_{pA}}
\newcommand\arobase{\alpha} % tresse bloquante
\newcommand\init[1]{S(#1)} % starting set
\newcommand\final[1]{F(#1)} % finishing set
\newcommand\g{\lambda}
\newcommand\m{\mu}
\newcommand\lc{\ell_c} % longueur canonique
\newcommand\pgcd[2]{#1 \wedge #2}
\newcommand\card[1]{\#(#1)}

\newcommand\boule[1]{\mathbf{B}(#1)}

\DeclareMathOperator\fn{NF_{\it l}}

\title[Genericity of pseudo-Anosov braids II]{On the genericity of pseudo-Anosov braids II:\\ conjugations to rigid braids} 
\author{Sandrine Caruso and Bert Wiest}
\date{}
\begin{document}
\begin{abstract}
We prove that generic elements of braid groups are pseudo-Anosov, in the following sense: in the Cayley graph of the braid group with $n\geqslant 3$ strands, with respect to Garside's generating set, we prove that the proportion of pseudo-Anosov braids in the ball of radius~$l$ tends to~$1$ exponentially quickly as $l$ tends to infinity. Moreover, with a similar notion of genericity, we prove that for generic pairs of elements of the braid group, the conjugacy search problem can be solved in quadratic time. The idea behind both results is that generic braids can be conjugated ``easily'' into a rigid braid.
%We prove that, in the $l$-ball of the Cayley graph of the braid group with $n\geqslant 3$ strands, the proportion of pseudo-Anosov braids tends to 1 exponentially quickly as $l$ tends to infinity.
\end{abstract}
\maketitle 

\section{Introduction}

%\marg{vérifier partout : décroissance exp *vers 0*. Vérifier aussi partout : $\inf=\epsilon=0,1$ ou $\inf=i\in\Z$ ? On veut $\Z$.}
In the recent article \cite{CarusoPA}, S.\ Caruso proved the following result. For a fixed number of strands $n$, consider the ball of radius~$l$ and center $1$ in the Cayley graph of the braid group $\B_n$, with generators the simple braids. Then for sufficiently large $l$, among the elements of this ball, the proportion of pseudo-Anosov braids is bounded below by a positive constant which does not depend on $l$ (but it might depend on $n$).
A key lemma in this paper states that among the \emph{rigid} braids with canonical length equal to~$l$, the proportion of pseudo-Anosov braids tends to~1 as $l$~tends to infinity. 

The aim of the present paper is to prove the following stronger result:

\begin{reptheorem}{T:main}
Consider the ball of radius~$l$ and center $1$ in the Cayley graph of the braid group~$\B_n$, with generators the simple braids. Then the proportion of pseudo-Anosov braids among the elements of this ball tends to~1 as $l$ tends to infinity. Moreover, this convergence happens exponentially fast. 
\end{reptheorem}

In fact, we shall prove a slightly stronger technical result: in the statement of the theorem, one can replace ``pseudo-Anosov braids'' by ``braids which admit a non-intrusive conjugation to a rigid pseudo-Anosov braid''.

The plan of the article is as follows: in Section~\ref{definitions}, we recall some classical definitions. In Section~\ref{nonintrusives}, we state the fact that, among braids with a fixed infimum, the proportion of those admitting a non-intrusive conjugation to a rigid braids tends to~1 exponentially quickly as the canonical length tends to infinity. This fact will be proven in Section~\ref{bloquante}, using the notion of a blocking braid. We complete the proof of the main theorem in Section~\ref{genericite}. In Section~\ref{S:ConjugRapide} we prove a related result, namely that the conjugacy problem in braid groups has generically a fast solution. Finally, we present some other consequences and conjectures arising from our results and techniques in Section~\ref{consequences}.

\section{Definitions}\label{definitions}

We recall that the Nielsen-Thurston classification theorem states that every element of $B_n$ is exactly one of the following: periodic, or reducible non-periodic, or pseudo-Anosov. In the context of braid groups, we must use the following definition of \emph{periodic}: a braid $x\in\B_n$ is periodic if and only if there exist non-zero integers $m$ and $l$ such that $x^m=\Delta^l$, where $\Delta = (\sigma_1 \cdots \sigma_{n-1})(\sigma_1\cdots\sigma_{n-2}) \cdots (\sigma_1\sigma_2)\sigma_1$. (Geometrically, $\Delta$~corresponds to a half-twist along the boundary of the disk. The center of~$\B_n$ is generated by the full twist~$\Delta^2$.)

\bigskip

We will also use some elements of Garside theory, in the classical case of braid groups, which we recall now. For more details, the reader can consult \cite{EM}, or \cite{DDGM} for the general theory. 

The group $\B_n$ is equipped with a partial order relation~$\preccurlyeq$, defined as follows: $x \preccurlyeq y$ if and only if $x^{-1}y \in \B_n^+$, the monoid of positive braids (i.e.\ only positive crossings). If $x \preccurlyeq y$, we say that $x$ is a \emph{prefix} of~$y$. Any two elements $x,y\in \B_n$ have a unique greatest common prefix, denoted $\pgcd x y$.

Similarly we define $\succcurlyeq$ as follows: $x \succcurlyeq y$ if and only if $x y^{-1} \in \B_n^+$. Notice that $x \succcurlyeq y$ is not equivalent to $y \preccurlyeq x$. If $x \succcurlyeq y$, we say that $y$~is a \emph{suffix} of~$x$.

The elements of the set $\{x \in \B_n, 1 \preccurlyeq x \preccurlyeq \Delta\}$ are called \emph{simple braids}, or \emph{permutation braids}.
Throughout this paper, we shall use the set of simple braids as the generating set of $\B_n$. The ball of radius~$l$ and center $1$ in the Cayley graph of~$\B_n$ with respect to this generating set will be denoted~$\boule l$.

\begin{defi}[left-weighting]
Let $s_1$, $s_2$ be two simple braids in $\B_n$. We say that $s_1$ and $s_2$ are \emph{left-weighted}, or that the pair $(s_1,s_2)$ is left-weighted, if there does not exist any generator $\sigma_i$ such that $s_1 \sigma_i$ and $\sigma_i^{-1} s_2$ are both still simple.
\end{defi}

\begin{defi}[starting set, finishing set]
Let $s \in \B_n$ be a simple braid. We call \emph{starting set of $s$} the set $\init s = \{i, \ \sigma_i \preccurlyeq s\}$ and \emph{finishing set of $s$} the set $\final{s} = \{i, \ s \succcurlyeq \sigma_i \}$.
\end{defi}

\begin{remark}
Two simple braids $s_1$ and $s_2$ are left-weighted if and only if $\init{s_2} \subset \final{s_1}$.
\end{remark}

\begin{prop}
Let $x \in \B_n$. There exists a unique decomposition $x = \Delta^p x_1\cdots x_r$ such that $x_1, \ldots,x_r$ are simple braids, distinct from $\Delta$ and $1$, and such that the pairs $(x_i,x_{i+1})$ are left-weighted for all $i = 1, \ldots, r-1$.
\end{prop}

\begin{defi}[left normal form]
In the previous proposition, the writing $x = \Delta^p x_1 \cdots x_r$ is called the \emph{left normal form} of $x$, $p$ is called the \emph{infimum} of $x$ and is denoted by $\inf(x)$, $p+r$ is the \emph{supremum} of $x$ and is denoted by $\sup(x)$, and $r$ is called the \emph{canonical length} of~$x$, and denoted~$\lc(x)$.

Furthermore, if $r \geqslant 1$, we denote by $\iota(x) = \tau^{-p}(x_1)$ 
%$\Delta^{p} x_1 \Delta^{-p}$ 
the \emph{initial factor} of $x$, where $\tau$ denotes the conjugation by~$\Delta$, i.e.\ $\tau(x)=\Delta^{-1} x \Delta$. (In particular $\iota(x) = x_1$ if $p$ is even, $\iota(x) = \Delta x_1 \Delta^{-1}$ if $p$ is odd.) We denote $\phi(x) = x_r$ the \emph{final factor} of~$x$.
%\marg{S: As $\Delta^2$ commute with all braids. B: Rien changé ici, mais rajouté ``The center of $B_n$ is generated by $\Delta^2$'' dans le 1er para de cette section}
\end{defi}

%\marg{Added the definition of rigidity here}
\begin{defi}[rigidity]
A braid $x$ of positive canonical length is said to be \emph{rigid} if the pair $(\phi(x),\iota(x))$ is left-weighted.
\end{defi}

Finally, we mention that at several key points in the present paper we shall use the article~\cite{CarusoPA}, and particularly its asymptotic estimates. For two number sequences~$(u_l)$ and $(v_l)$, we say that $u_l$ is of the order of $v_l$ if the sequences $(\frac{u_l}{v_l})$ and $(\frac{v_l}{u_l})$ are bounded.

\section{Non-intrusive conjugations}\label{nonintrusives}

\begin{definition}\label{defi:nonintrusive}
Let $x$ be a braid with normal form $x=\Delta^{\inf(x)}x_1\cdot\dots\cdot x_l$. A conjugation of $x$ is \emph{non-intrusive} if the normal form of the conjugated braid contains the subword  $x_{2 \cdot \lceil\frac{l}{5}\rceil+1}\cdots x_{l-2\cdot\lceil\frac{l}{5}\rceil}$.
\end{definition}

In other words, a conjugation of $x$ is non-intrusive if the middle fifth of the normal form of~$x$ still appears in the normal form of the conjugate.
%Autrement dit, une conjugaison est non-intrusive si le cinquième central de la tresse~$x$ apparaît encore dans la forme normale du conjugué.

\begin{example}\label{ex:conjnonintr}
Let~$x$ be the following braid with 4 strands and of canonical length~5:
$$x = \sigma_2 \sigma_3 \sigma_2 \sigma_1 \cdot \sigma_1 \sigma_3 \sigma_2 \sigma_1 \cdot \sigma_1 \sigma_2 \sigma_1 \sigma_3 \sigma_2 \cdot \sigma_3 \sigma_2 \sigma_1 \sigma_3 \cdot \sigma_1 \sigma_3 \sigma_2 \sigma_1.$$ 
Its middle fifth consists of the single factor $\sigma_1 \sigma_2 \sigma_1 \sigma_3 \sigma_2$. Let $\tilde x$ be its conjugate by the last two factors $\sigma_3 \sigma_2 \sigma_1 \sigma_3 \cdot \sigma_1 \sigma_3 \sigma_2 \sigma_1$ :
\begin{align*}
\tilde x &= \sigma_3 \sigma_2 \sigma_1 \sigma_3 \cdot \sigma_1 \sigma_3 \sigma_2 \sigma_1 \cdot \sigma_2 \sigma_3 \sigma_2 \sigma_1 \cdot \sigma_1 \sigma_3 \sigma_2\sigma_1 \cdot \sigma_1 \sigma_2 \sigma_1 \sigma_3 \sigma_2 \\&= \Delta \cdot \sigma_1 \sigma_2 \sigma_3 \sigma_1 \cdot \sigma_1 \sigma_3 \cdot \sigma_1 \sigma_3 \sigma_2 \sigma_1 \cdot \sigma_1 \sigma_2 \sigma_1 \sigma_3 \sigma_2 \text{ \ \ \ (in normal form)}\end{align*}
The conjugation from $x$ to $\tilde x$ is non intrusive, because $\tilde x$~contains the factor $\sigma_1\sigma_2 \sigma_1 \sigma_3 \sigma_2$ in its normal form.
\end{example}

\begin{notation}
We denote
$$
\B_n^{\epsilon,l}=\{x\in \B_n \ | \  \inf(x)=\epsilon, \lc(x)=l \}
$$
and $\nintr{(\epsilon,l)}$ the proportion, among the elements of $\B_n^{\epsilon,l}$, of braids which admit a non-intrusive conjugation to a rigid braid.
\end{notation}

We observe that for every $l\in\N$ and $\epsilon\in\Z$ we have $\nintr{(\epsilon,l)}=\nintr{(\epsilon+2,l)}$ -- thus $\nintr{(\epsilon,l)}$ depends only on~$n$, on~$l$, and on the parity of~$\epsilon$.

\begin{prop}\label{P:ConjNonintr}
There exists a constant $\muR \in (0,1)$ (which depends on~$n$) such that $\nintr{(\epsilon,l)}\geqslant 1-\muR^l$. 
\end{prop}

The aim of the next section is to prove this proposition.

\section{Blocking braids and the proof of Proposition~\ref{P:ConjNonintr}}\label{bloquante}

\begin{notation}
If $X$ and $Y$ are two braids, and if $Y$ is of infimum~$0$, then we denote $\fn(X \cdot Y)$ the word in (left) normal form representing the product $X\cdot Y$.

We say that $X\cdot Y$ \emph{is in normal form} if $\fn(X \cdot Y)$ is equal, as a word, to the normal form of $X$, followed by the normal form of $Y$.

If $s_1$ is the last factor of the normal form of~$X$, and $s_2$ is the first factor of the normal form of~$Y$, we are going to denote $\final{X} = \final{s_1}$ and $\init{Y} = \init{s_2}$. In particular, $X \cdot Y$ is in normal form if and only if $\init{Y} \subset \final{X}$.
\end{notation}

Let~$x$ be a braid of infimum $\epsilon\in\Z$, and of canonical length $l \geqslant 5$. We introduce some more notation. We cut the normal form representative of~$x$ (other than the initial power of~$\Delta$) into five pieces of roughly equal size, each of them in normal form:
%On découpe le mot en forme normale représentant $x$ (autre que la puissance de $\Delta$ initiale) en cinq morceaux de taille à peu près égale, chacun en forme normale à gauche~:
\begin{gather*}
P_1(x)=x_1\cdots x_{\lceil\frac{l}{5}\rceil}, \quad P_2(x)=x_{\lceil\frac{l}{5}\rceil+1}\cdots x_{2\cdot\lceil\frac{l}{5}\rceil}, \\
P_3(x)=x_{2\cdot\lceil\frac{l}{5}\rceil+1}\cdots x_{l-2\cdot\lceil\frac{l}{5}\rceil}, \\
P'_4(x)=x_{l-2\cdot\lceil\frac{l}{5}\rceil+1}\cdots x_{l-\lceil\frac{l}{5}\rceil}, \quad P'_5(x)=x_{l-\lceil\frac{l}{5}\rceil+1}\cdots x_l.
\end{gather*}
Notice that $P_1(x), P_2(x), P'_4(x)$ and $P'_5(x)$ have exactly equal length. The word $P_3(x)$ is the ``middle fifth'' subword mentioned in the previous section. Finally, we denote
$$P_4(x)=\tau^{\epsilon} (P'_4(x)) \text{ \ et \ } P_5(x)=\tau^{\varepsilon} (P'_5(x)).$$
If there is no ambiguity, we shall simply write $P_i$ instead of $P_i(x)$.
The braid~$x$ can always be conjugated to
$$
\tilde{x}=\Delta^\epsilon \cdot P_4\cdot P_5\cdot P_1\cdot P_2\cdot P_3
$$
and this writing is almost in normal form: the only place where two successive letters are not necessarily left-weighted is the transition from the last letter of~$P_5$ to the first letter of~$P_1$. All other pairs of successive letters are left-weighted, even $\phi(P_3)$ followed by $\iota(\Delta^\varepsilon P_4)$ (the last letter followed by the first). 
For this reason we also have $\iota(P_4)=\iota(P_4\cdot P_5)$ and $\phi(P_1\cdot P_2)=\phi(P_2)$.

\begin{observation}\label{O:RigideCritere}
Consider the normal form of~$P_4\cdot P_5\cdot P_1\cdot P_2$. If 
\begin{equation}\label{eq:O:RigideCritere:iota}
\iota(P_4\cdot P_5\cdot P_1\cdot P_2)=\iota(P_4\cdot P_5) % x_{k+1-2\cdot\lceil\frac{k}{5}\rceil}
\end{equation}
and
\begin{equation}\label{eq:O:RigideCritere:phi}
\phi(P_4\cdot P_5\cdot P_1\cdot P_2)=\phi(P_1\cdot P_2)
\end{equation}
then the braid~$\tilde x$ is non-intrusively conjugate to~$x$ (because the normal form of~$\tilde x$ will contain $P_3$ as a subword), and it is rigid.
\end{observation}
Intuitively, the hypothesis of Observation~\ref{O:RigideCritere} is that the given word representing~$\tilde x$ may not quite be in normal form, but that the modifications necessary in order to transform it into normal form are confined inside the word, and do not touch its extremities (up to a possible appearance of some factors $\Delta$, and up to conjugation of the initial factors of~$\tilde x$ by these factors~$\Delta$.) 

For instance, in Example~\ref{ex:conjnonintr}, the hypotheses of Observation~\ref{O:RigideCritere} are satisfied, and the conjugate $\tilde x$ is indeed rigid.

Our aim now is to prove that the proportion of braids~$x$ for which the hypotheses of Observation~\ref{O:RigideCritere} are satisfied tends to~1 when the length of~$x$ tends to infinity. In order to achieve this, we are going to observe that certain braids ``block the chain reaction of the transformation into normal form'', and that these ``blocking braids''\ have excellent chances of actually appearing.

\begin{figure}[htb]
\centerline{\includegraphics[scale=0.73]{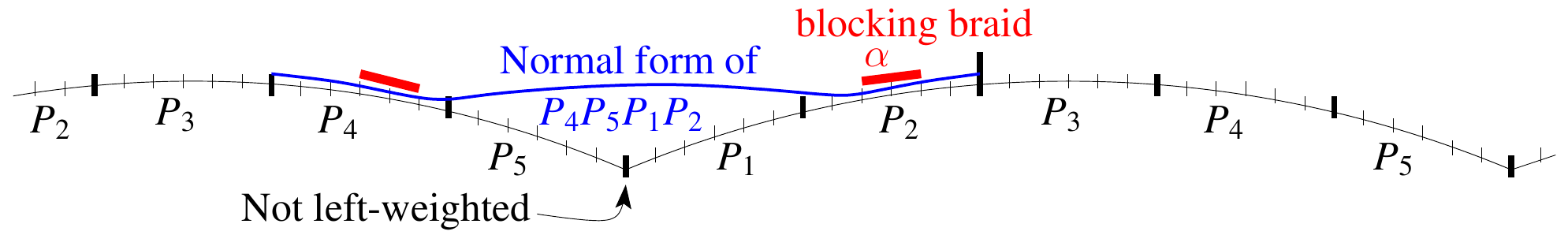}}
\caption{The stategy of the proof: this picture takes place in the Cayley graph of~$\B_n$. The braid~$x$ lifts to a bi-infinite path. The picture shows the generic situation: the last factor of the normal form of $P_4P_5P_1P_2$ coincides with the last factor of~$P_2$, and its initial factor (except for $\Delta$) coincides with the first factor of~$P_4$.} \label{F:Stratégie}
\end{figure}

We recall that for a simple braid~$s$, the \emph{complement} $\partial s$ is the braid $\partial s = s^{-1} \Delta$. We extend this definition to arbitrary braids~$y$ by the formula 
$$\partial y=y^{-1}\cdot \Delta^{\sup(y)}.$$
This is the unique braid such that $y\cdot \partial y=\Delta^{\sup(y)}$. 
%Observons que $\partial$ est presque une involution~: $\partial\partial y=\tau^{-\sup(y)}y$
If the normal form of~$y$ is $\Delta^{\inf y} y_1\cdots y_l$ then the normal form of $\partial y$ is $\overline y_l\cdots \overline y_1$, where $\overline y_{l-i}=\tau^{-i}(\partial y_{l-i})$ for $i=0,\ldots, l-1$ (i.e.\  $\overline y_{l-i}=y_{l-i}^{-1}\cdot \Delta=\partial y_{l-i}$ if $i$~is even and $\overline y_{l-i}=\Delta\cdot y_{l-i}^{-1}=\tau^{-1}(\partial y_{l-i})$ if $i$~is odd). In particular, $\inf(\partial y)=0$ and $\sup(\partial y)=\lc(y)$. 
We also calculate, for later reference, that
\begin{equation}\label{E:PhiCompl}
\phi(\partial y)=\tau^{-\sup(y)+1}(\partial\iota(y))
\end{equation}
because $\phi(\partial y)=\overline y_1=\tau^{-l+1}(\partial y_1)=\tau^{-l+1}(\partial \tau^{-\inf(y)}\iota(y))=\tau^{-\inf(y)-l+1}(\partial\iota(y))$.

Now, the normal form representative of $P_4\cdot P_5\cdot P_1\cdot P_2$ is 
\begin{equation}\label{eq:FNP4P5P1P2}
\fn\left(P_4\cdot P_5\cdot P_1\cdot P_2^{\phantom{1}}\right) = \fn\left(P_4\cdot P_5^{\phantom 1}\cdot t\right)\ \cdot\fn\left(t^{-1}\cdot P_1\cdot P_2\right)
\end{equation}
where
$$
t=\pgcd{(P_1\cdot P_2)}{\partial(P_4\cdot P_5)}.
$$
We also notice that, since $P_4 \cdot P_5 \cdot t \cdot t^{-1} \cdot \partial(P_4 \cdot P_5) = \Delta^{\sup(P_4 P_5)} = \Delta^{\sup(P_4 P_5 t)}$, the following formula holds:
\begin{equation}\label{eq:pgcdFN}
\partial(P_4\cdot P_5\cdot t) = t^{-1}\cdot \partial(P_4\cdot P_5).
\end{equation}
This suggests a way of studying the normal form of $P_4\cdot P_5\cdot P_1\cdot P_2$ in which $P_1\cdot P_2$ and $\partial(P_4\cdot P_5)$ play strictly symmetric roles: 

\begin{lem}\label{L:RigideCritereSym}
Still denoting $t=\pgcd{(P_1\cdot P_2)}{\partial(P_4\cdot P_5)}$, suppose that
\begin{equation}\label{eq:tP1P2}
\phi(t^{-1}\cdot P_1\cdot P_2)=\phi(P_1\cdot P_2),
\end{equation}
and
\begin{equation}\label{eq:tP4P5}
\phi(t^{-1}\cdot \partial(P_4\cdot P_5))=\phi(\partial(P_4\cdot P_5)).
\end{equation}
Then the hypotheses of Observation~\ref{O:RigideCritere} are satisfied, and the braid~$x$ is non-intrusively conjugate to a rigid braid.
\end{lem}

\begin{proof}
Let us suppose that \eqref{eq:tP1P2} holds. Then so does~\eqref{eq:O:RigideCritere:phi}, because
$$\varphi(P_4\cdot P_5 \cdot P_1 \cdot P_2) \stackrel{\eqref{eq:FNP4P5P1P2}}{=} \varphi(t^{-1} \cdot P_1 \cdot P_2)\stackrel{\eqref{eq:tP1P2}}{=}\varphi(P_1 \cdot P_2)$$
%=\varphi(P_2)$$ where the last equality follows from the fact that $P_1 \cdot P_2$ is in normal form. Thus we have~\eqref{eq:O:RigideCritere:phi}.

Let us now prove the implication from~\eqref{eq:tP4P5} to \eqref{eq:O:RigideCritere:iota}. Assuming \eqref{eq:tP4P5}, we calculate
\begin{multline*}
\tau^{\sup(P_4\cdot P_5\cdot t)-1}(\partial \iota(P_4\cdot P_5\cdot t)) \stackrel{\eqref{E:PhiCompl}}{=} \phi(\partial(P_4\cdot P_5\cdot t)) = \phi(t^{-1}\cdot\partial(P_4\cdot P_5))\stackrel{\eqref{eq:tP4P5}}{=}\\
\stackrel{\eqref{eq:tP4P5}}{=} \phi(\partial(P_4\cdot P_5)) = \tau^{\sup(P_4\cdot P_5)-1}(\partial\iota(P_4\cdot P_5)).
\end{multline*}
Since $\sup(P_4\cdot P_5\cdot t)=\sup(P_4\cdot P_5)$, this implies $\iota(P_4\cdot P_5\cdot t) = \iota(P_4\cdot P_5)$, i.e.\ \eqref{eq:O:RigideCritere:iota}.
\end{proof}

Our aim now is to show that, in most cases, \eqref{eq:tP1P2} and \eqref{eq:tP4P5} are indeed satisfied.

\begin{definition}
A positive braid $\arobase$ is a \emph{blocking braid} if there exists an $i\in \{1,...,n-1\}$ so that for each braid~$X$ with $\inf(X)=0$ 
such that $X\cdot \arobase$ is in left normal form, the only non trivial simple braid which is a suffix of $X\cdot \arobase$ is $\sigma_i$. In other words, the last factor of the \emph{right} normal form of $X\cdot \arobase$ must be~$\sigma_i$.
\end{definition}

\begin{lemma}\label{L:XarobaseY}
Let $\arobase$ be a blocking braid and $X$ be a braid such that $\inf X = 0$ and such that $X\cdot\arobase$ is in normal form. Let $t$ be a prefix of $X\cdot\arobase$. If $(\sigma_i=)\phi(X\cdot \arobase) \neq \phi(t^{-1}\cdot X\cdot \arobase)$ then $t=X\cdot \arobase$.
\end{lemma}

\begin{proof} 
Let $s=t^{-1}\cdot X\cdot \arobase$ be the braid such that $t\cdot s=X\cdot \arobase$ (of course, $t\cdot s$ is not in normal form as written). Let us suppose (to obtain a contradiction) that $s$ is nontrivial. Then $\phi(s)$ is a nontrivial simple braid which is a suffix of $s$ and so of $t\cdot s=X\cdot \arobase$. Yet, by hypothesis, the only nontrivial simple braid which is a suffix of $X\cdot \arobase$ is~$\sigma_i$. So $\phi(X \cdot \arobase) = \phi(s)$: contradiction.
\end{proof}

\begin{lemma}\label{L:ProprBloqu} Let $\arobase$ be a blocking braid and let $X,Y$ be braids such that $\inf X =\inf Y= 0$ and such that $X\cdot\arobase\cdot Y$ is in normal form. Let $t$ be a prefix of $X\cdot\arobase\cdot Y$.  If $\phi(t^{-1}\cdot X\cdot \arobase\cdot Y)\neq \phi(X\cdot \arobase\cdot Y)$, 
%\marg{Confusing notation: using the word ``prefix'' for actual words here, not elements of the group}
then the normal form of $t$ contains the normal form of $X\cdot \arobase$ as a prefix.
\end{lemma}

\begin{proof} 
Let $t_1 = \pgcd{t}{(X \cdot \arobase)}$. We claim, and will prove below, that $\phi(t_1^{-1} \cdot X \cdot \arobase) \neq \phi( X \cdot \arobase)$. By applying Lemma~\ref{L:XarobaseY} to $t_1$, we deduce that $X \cdot \arobase$ is a prefix of $t$. 
%\marg{Dans ce paragraphe on ne re-utilise pas le fait que $\arobase$ est bloquante -- c'est bien.}
It remains to show that the normal form of $X \cdot \arobase$ is even the beginning of the normal form of $t$: indeed, $t$ is a prefix of $X\cdot\arobase\cdot Y$ and so $(X \cdot \arobase)^{-1} t$ is a prefix of $Y$. In particular, $\init{(X \cdot \arobase)^{-1} t} \subset \init Y \subset \final{\arobase}$, the last inclusion coming from the fact that $X\cdot\arobase\cdot Y$ is in normal form. So $X \cdot \arobase \cdot \fn((X \cdot \arobase)^{-1} t)$ is in normal form, which implies what we wanted.

Here is now the proof of our claim. We suppose, for a contradiction, that $\phi(t_1^{-1} \cdot X \cdot \arobase) = \phi( X \cdot \arobase)$. This means that $\fn(t_1^{-1} \cdot X \cdot \arobase) \cdot Y$ is in normal form. Since $\pgcd{t_1^{-1} t}{t_1^{-1} X \cdot \arobase} = 1$, we can deduce that we also have $\pgcd{t_1^{-1} t}{t_1^{-1} X \cdot \arobase \cdot Y} = 1$. Then, as $t_1^{-1} t$ left-divides $t_1^{-1} X \cdot \arobase \cdot Y$, we have $t_1^{-1} t = 1$. Finally, this implies that $t$ is a prefix of $X \cdot \arobase$, and so, by Lemma~\ref{L:XarobaseY}, $\phi(t^{-1}\cdot X\cdot \arobase\cdot Y) = \phi(X\cdot \arobase\cdot Y)$, contradicting the hypothesis.
\end{proof}

\begin{lemma}\label{L:BloquExiste}
Blocking braids exist.
\end{lemma}

\begin{proof}
Here is such a construction: denoting by $\Delta_{i,j}$ the positive half-twist involving the strands $i,i+1,\ldots ,j$, let
\begin{multline*}
\arobase = \Delta_{1,n-1} \sigma_{n-1} \cdot \Delta_{1,n-2} \sigma_{n-1} \sigma_{n-2} \cdot \Delta_{1,n-3} \sigma_{n-2} \sigma_{n-3} \cdot \Delta_{1,n-4} \sigma_{n-3} \sigma_{n-4} \cdots \\
\sigma_1\sigma_2\sigma_1 \sigma_4 \sigma_3 \cdot \sigma_1 \sigma_3 \sigma_2 \cdot \sigma_2. 
\end{multline*}

\begin{figure}[htb]
\begin{center} \includegraphics[height=9cm]{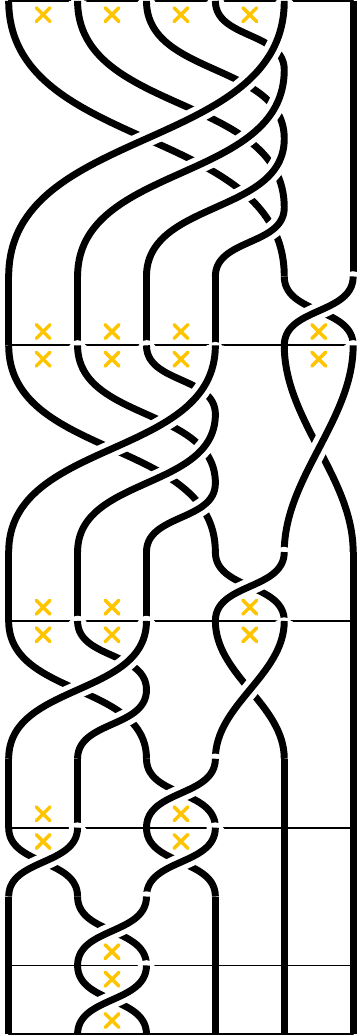}\end{center}
\caption{Example of a blocking braid with $6$ strands. The yellow crosses indicate the starting and finishing sets.}\label{F:blocking}
\end{figure}

For example with $6$ strands (see Figure~\ref{F:blocking}):
$$
\arobase = \sigma_1 \sigma_2 \sigma_3 \sigma_4 \sigma_1 \sigma_2 \sigma_3 \sigma_1 \sigma_2 \sigma_1 \sigma_5 \cdot \sigma_1 \sigma_2 \sigma_3 \sigma_1 \sigma_2 \sigma_1 \sigma_5 \sigma_4 \cdot \sigma_1 \sigma_2 \sigma_1 \sigma_4 \sigma_3 \cdot \sigma_1 \sigma_3 \sigma_2 \cdot \sigma_2.
$$

It is a braid word which is in left normal form, but also in right normal form. We observe that the starting set of~$\arobase$ is 
$\init \arobase=\{1,...,n-2\}$ and its finishing set is
$\final \arobase=\{2\}$. If $X\cdot \arobase$ is in (left) normal form, then $\final X\supseteq \{1,...,n-2\}$ and so $\final X=\{1,...,n-2\}$. This implies that $X\cdot\arobase$ is also in right normal form. So the only simple factor which can be extracted by the right from $X\cdot \arobase$ is $\sigma_2$, as we wanted.  
\end{proof}

In order to prove that blocking braids are almost certain to occur just where we need them, we will use the following lemma, which results from Lemma 3.5 and Remark 3.6 in \cite{CarusoPA}.

\begin{lem}\label{lem:sousmot}
Let $a_1$, $a_2$, $a_3\co \N\to\N$ be functions with $a_1+a_3$ and $a_2$ non-decreasing and tending to infinity, %when $l$ tends to infinity, 
and such that $a_1(l) + a_2(l) + a_3(l) = l$.
For each braid $x$ of infimum $\epsilon \in \Z$ and of canonical length $l$, of normal form $\Delta^\epsilon x_1 \cdots x_l$, denote by $P(x) = x_{a_1(l)+1} \cdots x_{a_1(l)+a_2(l)}$ (so $P(x)$ is a part of the normal form of $x$ of length $a_2(l)$).

Let $w$ be a fixed braid. Then the proportion of braids $x \in \B^{\epsilon, l}_n$ such that the normal form of $P(x)$ contains that of $w$ as a subword tends exponentially quickly to $1$ when $l$ tends to infinity.
\end{lem}

\begin{proof}[Proof of Proposition~\ref{P:ConjNonintr}]
We recall that we have to prove that the proportion, among the braids $x$ in~$\B_n^{\epsilon,l}$, of braids for which one of the two hypotheses,
\eqref{eq:tP1P2} or \eqref{eq:tP4P5}, of Lemma~\ref{L:RigideCritereSym} is \emph{not} satisfied, tends exponentially quickly to~0 as $l$ tends to infinity. 
In fact, we shall only prove that braids not satisfying hypothesis~\eqref{eq:tP1P2} are rare. Since the operation of taking the complement $\partial\co \B_n^{0,2\lceil \frac l 5 \rceil}\to \B_n^{0,2\lceil \frac l 5 \rceil}$ is a bijection, we have a completely analogue situation for hypothesis~\eqref{eq:tP4P5}. 

By Lemma~\ref{lem:sousmot}, the proportion of braids~$x$ such that $P_2(x)$ contains a blocking braid tends to~1 exponentially quickly. 
Among these braids, look at those for which 
$$\phi(P_4\cdot P_5\cdot P_1\cdot P_2)\neq \phi(P_1\cdot P_2)$$
holds, or in other words
$$\phi(t^{-1}\cdot P_1\cdot P_2)\neq \phi(P_1\cdot P_2) \text{, \ \ where \ \ }t=\pgcd{(P_1\cdot P_2)}{\partial(P_4\cdot P_5)}
$$
For those braids, by Lemma~\ref{L:ProprBloqu}, the normal form of $t$ must contain that of $P_1$ as a prefix,
%Among these braids, according to Lemma~\ref{L:ProprBloqu}, a necessary condition for 
%$$\phi(P_4\cdot P_5\cdot P_1\cdot P_2)\neq \phi(P_1\cdot P_2),$$
%and so for $\phi(t^{-1}\cdot P_1\cdot P_2)\neq \phi(P_1\cdot P_2)$, is that the normal form of $t$ contains that of $P_1$ as a prefix, 
%\marg{Rewritten - please check. I'm quite sure we don't need the full strength of Lemma 4.6. We do need, however, the full strength of Lemma 4.9!}
and in particular $P_1\preccurlyeq t$. (Intuitively, the factor $P_1$ must be completely ``eaten''\ during the transformation of $P_4\cdot P_5\cdot P_1\cdot P_2$ into normal form, possibly creating some new factors $\Delta$.) Thus
$$P_1 = \pgcd{P_1}{\Delta^{\lceil\frac{l}{5}\rceil}} \preccurlyeq \pgcd{t}{\Delta^{\lceil\frac{l}{5}\rceil}} \preccurlyeq \pgcd{(\partial(P_4\cdot P_5))}{\Delta^{\lceil\frac{l}{5}\rceil}} = \partial P_5
$$
So $P_1(x)$ must be a prefix of $\partial P_5(x)$. Yet, the proportion of braids $x$ for which this is the case is negligible:

\begin{lem}\label{L:PrefixRare}
The proportion, among all elements of $\B_n^{\varepsilon,l}$, of braids $x \in \B_n^{\varepsilon,l}$ such that $P_1$ is a prefix of $\partial P_5$  decreases exponentially quickly to~$0$ when $l$~tends to infinity.
\end{lem}

\begin{proof}
We decompose $\partial P_5$ in two parts of length $\lceil \frac l {10} \rceil$ 
and $\lceil \frac l {5} \rceil-\lceil \frac l {10} \rceil$: $\partial P_5 = Q_1 \cdot Q_2$. As before, according to Lemma \ref{lem:sousmot}, the proportion of braids $x \in \B_n^{\varepsilon,l}$ such that $Q_2$ contains a blocking braid tends exponentially quickly to $1$ (more precisely, according to~\cite{CarusoPA}, the number of braids for which this is not the case is of the order of $\g^{l-\frac l{10}}\m^{\frac l{10}}$ for two constants $1 < \m < \g$, while the cardinality of $\B_n^{\varepsilon,l}$ is of the order of $\g^l$).

We now show the following: if $x$ satisfies the condition of the lemma that $P_1$ is a prefix of $\partial P_5$, and if $Q_2$ contains a blocking braid, then the normal form of $P_1$ contains that of $Q_1$ as a prefix. For that, it suffices to prove that $\phi(P_1^{-1} \partial P_5) \neq \phi(\partial P_5)$ and to apply Lemma \ref{L:ProprBloqu}.

Let us recall that $P_1$ and $\partial P_5$ have the same length $\lceil \frac l 5 \rceil$. To simplify the notations, let us denote by $k = \lceil \frac l 5 \rceil$, and by $P_1 = y_1 \cdots y_k$ and $\partial P_5 = z_1 \cdots z_k$ the normal forms. The condition of the lemma is that $y_1 \cdots y_k$ is a prefix of $z_1 \cdots z_k$. Let us suppose for a contradiction that $\phi(P_1^{-1} \partial P_5) = \phi(\partial P_5)$, \emph{i.e.} that $\phi(y_l^{-1} \cdots y_1^{-1} z_1 \cdots z_k) = z_k$. This means that
$$\fn(y_k^{-1} \cdots y_1^{-1} z_1 \cdots z_k) = \fn(y_k^{-1} \cdots y_1^{-1} z_1 \cdots z_{k-1}) \cdot z_k$$ 
and in particular that $y_k^{-1} \cdots y_1^{-1} z_1 \cdots z_{k-1}$ is a positive braid, \emph{i.e.} that $y_1 \cdots y_k$ is a prefix of $z_1 \cdots z_{k-1}$. This is impossible, as $y_1 \cdots y_k$ is a longer braid than $z_1 \cdots z_{k-1}$.

We deduce that, if $Q_2$ contains a blocking braid, under the condition that $P_1$ is a prefix of $\partial P_5$, then the normal form of $P_1$ contains that of $Q_1$ as a prefix. A braid $x$ satisfying these conditions is thus determined by at most $l-\lceil \frac l{10} \rceil$ factors, since the $\lceil \frac l{10} \rceil$ factors of $Q_1$ are determined by the first factors of $P_1$. So the proportion of such braids is, still according to \cite{CarusoPA}, of the order of $\g^{-\frac l {10}}$.

Finally, among all braids of $\B_n^{\varepsilon,l}$, the proportion of elements such that $P_1$ is a prefix of $\partial P_5$ decreases exponentially quickly to $0$ with $l$. This completes the proof of Lemma~\ref{L:PrefixRare}.
\end{proof}

\begin{proof}[Alternative proof of Lemma~\ref{L:PrefixRare}]
We recall that $P_1$ and $\partial P_5$ have the same length $\lceil \frac l 5 \rceil$. To simplify the notations, we denote by $k = \lceil \frac l 5 \rceil$, and by $P_1 = y_1 \cdots y_k$ and $\partial P_5 = z_1 \cdots z_k$ the normal forms. Assume that $y_1 \cdots y_k$ is a prefix of $z_1 \cdots z_k$ -- our aim is to show that this decreases substantially the number of possible words $y_1\cdots y_k$. 

For $i=1,\ldots,k-1$, the braid $y_1 \cdots y_i$ should be a prefix of $z_1 \cdots z_i$. Denote by $\delta_i$ the positive braid $y_i\inv \cdots y_1\inv z_1 \cdots z_i$. As $y_{i+1}y_{i+2}\delta_{i+2}=\delta_i z_{i+1}z_{i+2}$, it follows that $y_{i+1}y_{i+2}$ is a divisor of $\delta_i z_{i+1}z_{i+2}$; moreover, this last braid does not contain any $\Delta$-factor (because if it did, then so would $\partial P_5=y_1\cdots y_i\delta_i z_{i+1}\cdots z_k$.) This enforces a strong restriction on the possible factors $y_{i+1}\cdot y_{i+2}$, beyond the obvious requirement that $y_i\cdot y_{i+1}\cdot y_{i+2}$ should be in normal form. 

We will now use the fact (which we leave to the reader as an amusing exercise) that in every positive braid whose normal form contains exactly two factors, both %of which are 
different from~$\Delta$, there is a pair of strands that do not cross. In each divisor of such a braid the corresponding strands do not cross either. Let us apply this fact to the first two factors of $\delta_i z_{i+1}z_{i+2}$: there exists a pair of strands, the $r$th and the $s$th, that do not cross, and hence do not cross in $y_{i+1}y_{i+2}$, either. Let $t$ be an element of $F(y_i)$. We can then construct a braid in normal form $y_{i+1}\cdot y_{i+2}$ such that $I(y_{i+1})=\{t\}$ and where the $r$th and $s$th strands cross. 
%\marg{C'est plus clair comme ca avec exemple?}
(This is an easy exercise - for example, in $\B_6$ if $t=1$, $r=4$ and $s=6$, we can choose $y_{i+1}\cdot y_{i+2}=\sigma_1\sigma_2\sigma_3\cdot\sigma_3\sigma_4\sigma_5$; if $t=3$, $r=1$, $s=6$, we choose $y_{i+1}\cdot y_{i+2}$ so that $y_{i+1}=\sigma_3\sigma_2\sigma_1\sigma_4\sigma_3\sigma_2\sigma_5\sigma_4\sigma_3$). This choice for $y_{i+1}$ and $y_{i+2}$ is therefore forbidden by the hypothesis that $P_1$ is a prefix of $\partial P_5$, %although it is compatible with the requirement that $y_i\cdot y_{i+1}\cdot y_{i+2}$ should be in normal form. 
even though $y_1\cdots y_i\cdot y_{i+1}\cdot y_{i+2}$ is in normal form.

Since there is such a restriction for every value of~$i$ between $1$~and $k-2$ (and this for every possible braid $\partial P_5$), the set of braids for which $P_1$ is a divisor of $\partial P_5$ has a lower rate of exponential growth than that of all braids. This implies Lemma~\ref{L:PrefixRare}.
\end{proof}

The proof of Proposition~\ref{P:ConjNonintr} is now complete. Let us summarize again: among the braids~$x\in\B_n^{\epsilon,l}$, ``generic'' ones (a proportion which tends exponentially quickly to~1 as $l$~tends to~$\infty$) contain a blocking braid in their $P_2(x)$-segment (and, symmetrically, in $\partial P_4(x)$). For such a braid containing a blocking braid in $P_2$ and in $\partial P_4$, the only way to avoid being non-intrusively conjugate to a rigid braid is that, in the process of transforming $P_4 P_5\cdot P_1 P_2$ into normal form 
\begin{itemize}
\item either $P_1$ is completely absorbed into $P_5$, 
possibly creating some new factors $\Delta$
\item or, symmetrically, $P_5$ is completely absorbed into $P_1$. 
\end{itemize}
As seen in Lemma~\ref{L:PrefixRare}, generically this does not happen (it only happens to a proportion of braids which tends exponentially quickly to 0).
%By the symmetric argument, it happens only exponentially rarely that 
%\marg{symmetry mentioned again}
%$$\iota(P_4\cdot P_5\cdot P_1\cdot P_2)\neq \iota(P_4).$$
%According to Observation~\ref{O:RigideCritere}, the proportion of braids $x$ such that $\tilde x$ is a rigid non-intrusive conjugate tends exponentially quickly to $1$ with the canonical length of $x$.
\end{proof}

\section{Pseudo-Anosov braids are generic}\label{genericite}

\begin{theorem}\label{T:main}
Consider the ball $\boule l$ of radius~$l$ and center $1$ in the Cayley graph of the braid group~$\B_n$, with generators the simple braids. Then the proportion of pseudo-Anosov braids among the elements of this ball tends to~1 as $l$ tends to infinity. Moreover, this convergence happens exponentially fast. 
\end{theorem}

Several key points of the proof come directly from~\cite{CarusoPA}.

\begin{lemma}\label{L:pAparmiBnik}
There exists a constant~$\mupA$ (which depends on~$n$) such that, among the braids in $\B_n^{\epsilon,l}$, the proportion of those that can be non-intrusively conjugated to a rigid pseudo-Anosov braid is at least $1-\mupA^l$ (for sufficiently large~$l$, independently of~$i$).
\end{lemma}

%Tout d'abord les seules tresses rigides périodiques sont les puissances de~$\Delta$. Concernant le cas rigide réductible, 

\begin{proof} 
Proposition 4.5 of the paper~\cite{CarusoPA} explains how two theorems, one due to Gonz\'alez-Meneses and Wiest, the other to Bernardete, Guttierez and Nitecki, can be used to prove that the normal form of a rigid braid which is not pseudo-Anosov satisfies some extremely restrictive conditions. 
%There are two possibilities: either every single factor sends a round curve to a round curve, or every single factor sends a round or almost round curve to a round or almost round curve. In the second case we have extra restriction that either  
%En combinant un théorème de Gonz\'alez-Meneses / Wiest, avec celui de Bernardete / Guttierez / Nitecki, on voit que chaque facteur de la forme normale d'une tresse rigide réductible doit soit envoyer une courbe ronde sur une courbe ronde, soit préserver une courbe presque ronde, de telle façon que tout couple de brins à l'intérieur de la courbe presque ronde ou bien ne se croise pas de tout, ou bien se croise dans chaque facteur. 
%\marg{Il faudra re-écrire pour tenir compte du cas $\inf=$impair. François trouve tout le paragraphe trop vague. Sandrine a une nouvelle version }
%
%L'article \cite{CarusoPA} (dans la proposition 4.4) profite de cette obstruction à la réductibilité de tresses rigides pour exhiber 
%trois sous-mots, \marg{En fait, ce n'est pas expliqué exactement comme ça dans \cite{CarusoPA}}d'une longueur de deux ou trois lettres chacun, qu'une tresse rigide ne doit pas contenir tous si elle veut être réductible. 
More precisely, there are two words in normal form, one of length~$2$, the other of length~$4$, with the following property: if the normal form of a rigid braid contains both of these words as subwords, then the braid is pseudo-Anosov.
%deux sous-mots, d'une longueur de deux ou quatre lettres chacun, qu'une tresse rigide ne doit pas contenir tous les deux, ou son complémentaire contenir le deuxième, si elle veut être réductible.

Let us consider the proportion, among the elements $x$ of~$\B_n^{\epsilon,l}$, of braids which contain in their middle fifth~$P_3(x)$ the two subwords mentioned in the previous paragraph. It follows from Lemma~\ref{lem:sousmot} 
%3.5 et de la remarque 3.6 dans \cite{CarusoPA} 
%couplée à la démonstration de la proposition 4.4) 
that this proportion tends to~$1$ exponentially quickly: there exists a constant~$\muC$ (which depends on $n$) such that this proportion is at least $1-\muC^l$. (The index $M$ in the  notation $\muC$ comes from the word ``middle''.)

We now look at the intersection of two subsets of $\B_n^{\epsilon,l}$:
\begin{enumerate}
\item The braids in $\B_n^{\epsilon,l}$ which can be non-intrusively conjugated to a rigid braid
\item The braids $x$ in $\B_n^{\epsilon,l}$ which contain, in their middle fifth~$P_3(x)$, the two subwords mentioned previously, which stop rigid braids from being reducible or periodic. (We insist that this second subset may well contain reducible braids, but none that are rigid \emph{and} reducible.)%no rigid reducible ones.)
\end{enumerate}

The braids belonging to this intersection are all pseudo-Anosov (in fact they are conjugate to rigid pseudo-Anosov braids). 
%Des tresses qui appartiennent à cette intersection ne peuvent pas être réductibles, par construction. En plus, les seules tresses rigides périodiques sont les puissances de $\Delta$ (voir \cite{CarusoPA}). Donc tous les éléments de cette intersection se conjuguent de façon non-intrusive vers une tresse rigide et pseudo-anosovienne.
Moreover, for $l>0$, %\marg{pour $l>0$ ou pour $l$ suffisamment grand?} 
the proportion of elements of $\B_n^{\epsilon,l}$ which belong to the first subset is at least $1-\muR^l$ by Proposition~\ref{P:ConjNonintr}, and for the second subset the proportion is bounded below by %et pour le deuxième sous-ensemble c'est au moins 
$1-\muC^l$. Hence the proportion of elements belonging to the intersection of the two is at least $1-\muR^l-\muC^l$. Thus for any $\mupA$ larger than $\max(\muR,\muC)$, we have the desired result.
This concludes the proof of Lemma~\ref{L:pAparmiBnik}.
\end{proof}

\begin{proof}[Proof of Theorem~\ref{T:main}]
%Nous sommes prêts à conclure la démonstration du theorème~\ref{T:main}.
%On va mélanger trois ingrédients.
We are going to use three ingredients.

Firstly, we recall from~\cite{CarusoPA} that there exists a number $\g>1$ (which depends on~$n$) with the property that $|\B_n^{\epsilon,k}|=\Theta(\g^k)$, meaning that the sequences $\frac{|\B_n^{\epsilon,k}|}{\g^k}$ and $\frac{\g^k}{|\B_n^{\epsilon,k}|}$ stay bounded as $k$~tends to infinity.

Secondly, as in~\cite{CarusoPA} (Section 4.3), we observe that $\boule l$, the ball of radius~$l$ and center~$1$ in the Cayley graph of~$\B_n$, is the disjoint union 
$$
\boule{l} \ = \ \bigcup_{k=0}^l \ \bigcup_{i=-l}^{l-k} \ \B_n^{\epsilon,k}
$$
(This observation hinges on the fact that braids in so-called \emph{mixed normal form} are geodesics, which is proven in~\cite{CM}.)

Thirdly, Lemma~\ref{L:pAparmiBnik} ensures that among the elements of every $\B_n^{\epsilon,k}$, the proportion of braids not admitting a non-intrusive conjugation to a rigid pseudo-Anosov braid is an $O(\mupA^k)$, for a certain number~$\mupA$ with $0<\mupA<1$. 

The last two ingredients together imply that the total number of braids in $\boule l$ which cannot be non-intrusively conjugated to a rigid pseudo-Anosov braid is a
$$
O\left((2l+1) + 2l\cdot(\g\cdot\mupA)^1 + (2l-1)\cdot(\g\cdot\mupA)^2 + \cdots + (2l-l+1)\cdot(\g\cdot\mupA)^l \right)
$$
Therefore, the proportion in~$\boule l$ of elements which cannot be non-intrusively conjugated to a rigid pseudo-Anosov braid is a
%La proportion, parmi les tresses dans la $l$-boule, de ceux qui ne sont pas non-intrusivement conjugués vers une tresse rigide, pseudo-anosovienne est donc un
$$
O\left(\frac{2l+1}{\lambda^l}+\frac{2l\cdot \mupA}{\lambda^{l-1}}+\frac{(2l-1)\cdot\mupA^2}{\lambda^{l-2}}+\ldots+\frac{(2l-l+1)\cdot\mupA^l}{1} \right)
$$
$$
\leqslant O\left((l+1)\cdot (2l+1)\cdot\left(\max (\textstyle{\frac{1}{\lambda}},\mupA )\right)^l\right)
$$
and thus, for any $\varepsilon>0$, a 
$$
O\left( \left(\max (\textstyle{\frac{1}{\lambda}},\mupA) +\varepsilon\right)^l \right).
$$
Choosing $\varepsilon$ so small that $\max (\textstyle{\frac{1}{\lambda}},\mupA ) + \varepsilon<1$ yields the result.
%$\g\cdot\mupA+\varepsilon<\g$, on voit que ce nombre est exponentiellement plus petit que le nombre total d'éléments de la $k$-boule, qui croît au moins comme~$\g^l$.
\end{proof}

\section{Fast solutions to the conjugacy problem}\label{S:ConjugRapide}

The aim of this section is to prove that ``generically, the conjugacy problem in~$B_n$ can be solved in quadratic time''. First we recall a standard method for solving the conjugacy problem in braid groups. In order to decide whether two given braids $x_1$ and $x_2$ are conjugate, one calculates a certain finite subset $SC(x_i)$ of the conjugacy class of~$x_i$, for $i=1,2$. We shall not need the precise definition of this subset, called the ``sliding circuit set'' $SC(x)$ of a braid~$x$, we only need to know two things about it:
\begin{itemize}
\item The set $SC(x)$ depends only on the conjugacy class of~$x$, and it is always non-empty.
\item If the conjugacy class of~$x$ contains a rigid braid, then $SC(x)$ consists precisely of the rigid conjugates of $x$ \cite{GGM}.
\end{itemize}

Now in order to decide whether $x_1$ and $x_2$ are conjugate, it suffices to test if an arbitrarily chosen element of $SC(x_1)$ is contained in $SC(x_2)$.

Our aim is to show that for a ``generic'' element $x$ of $\B_n$, we can calculate the set $SC(x)$ in polynomial time.

\begin{remark}\label{remarque:orbiteSC}
We remark that for a rigid braid $x_r$, the set of rigid conjugates $SC(x_r)$ contains at least the orbit of~$x_r$
\begin{itemize}
\item under $\tau$, i.e.\ under conjugation by $\Delta$, and
\item under cyclic permutation of the factors other than $\Delta$.
\end{itemize}
This orbit has at most $2\cdot \lc(x_r)$ elements.
We will see later that for a ``generic'' braid, the set of rigid conjugates contains exactly \emph{one} such orbit.
\end{remark}

\begin{theorem}\label{T:ConjugRapide}
There exists an algorithm which takes as its input a braid $x\in\B_n$, whose running time is $O(\lc(x)^2)$, and which outputs
\begin{enumerate}
\item either a rigid conjugate of~$x$, equipped with a certificate that the set of rigid conjugates of~$x$ contains only its orbit under the action of~$\tau$ and under cyclic permutation of the factors (other than $\Delta$),
\item or the answer ``I don't know''.
\end{enumerate}
Among the elements in the ball of radius~$l$ and center~$1$ in the Cayley graph of~$\B_n$, the proportion of braids in case {\rm (2)} tends to~$0$ exponentially fast as $l$~tends to infinity.
\end{theorem}

\begin{proof}
As in section~\ref{bloquante}, we cut the braid $x$ into $5$ pieces $P_1$, $P_2$, $P_3$, $P'_4$ and $P'_5$, and we denote $P_4 = \tau^{\inf x} (P'_4)$ and $P_5 = \tau^{\inf x}(P'_5)$. We denote $P_{12} = P_1 \cdot P_2$ and $P_{45} = P_4 \cdot P_5$. (In fact, in order to describe the algorithmic procedure, it would be sufficient to cut the braid into only $3$ pieces, but for explaining why the algorithm works it is more convenient to retain the notation of the previous sections.) Then we execute the following operations:
\begin{enumerate}
\item calculate $\fn(P_{45} P_{12})$ ;
\item test whether $\iota(P_{45} P_{12}) = \iota(P_{45})$. If this is false, answer ``I don't know'' and stop. If it is true, continue;
\item test whether $\phi(P_{45} P_{12}) = \phi(P_{12})$. If this is false, answer ``I don't know'' and stop. If it is true, continue;
\item test whether $P_3$ contains in its normal form the subword $(\Delta \sigma_2^{-1}) \cdot \sigma_1$. If this is false, answer ``I don't know'' and stop. If it is true, continue;
\item output $\Delta^{\inf x} P_{45} P_{12} P_3$.
\end{enumerate}
Tests (2) and (3) check whether the conditions of Observation~\ref{O:RigideCritere} hold for the braid~$x$. If they do, then the braid $y = \Delta^{\inf x} P_{45} P_{12} P_3$ is indeed a rigid conjugate of~$x$, and moreover there is a non-intrusive conjugation of $x$ to $y$. 
Let us now suppose that $x$ passes the test (4). Since the conjugation is non-intrusive, $y$ also contains the subword $(\Delta \sigma_2^{-1}) \cdot \sigma_1$. After a further cyclic permutation of the factors of $y$, we obtain a rigid braid~$z$ with $\iota(z) = \sigma_1$ and $\phi(z) = (\Delta \sigma_2^{-1})$, or possibly $\iota(z) =\sigma_{n-1}$ and $\phi(z)=\Delta \sigma_{n-2}^{-1}$.

We claim that under these circumstances the set $SC(z)$ consists only of the single orbit defined in Remark~\ref{remarque:orbiteSC}.
%Pour cela, on utilise le lemme 6.1 de \cite{GMW} qui affirme que si deux tresses $z_1$ et $z_2$ conjuguées sont dans l'ensemble des circuits glissants, il existe une suite de tresses $z_1 = \alpha_1, \alpha_2, \ldots, \alpha_{r+1} = z_2$ telles pour tout $i$, $\alpha_{i+1}$ est le conjugué de $\alpha_i$ par un préfixe de $\iota(\alpha_i)$ ou de $\partial\phi(\alpha_i)$. Or, 
The proof of this claim is essentially the same as the proof of Lemma 2.4 in\cite{CarusoSSS}: it suffices to prove that conjugating~$z$ by any strict prefix of $\iota(z)$ or of $\partial\phi(z)$ never yields an element of $SC(z)$. That, however, is a tautology: neither $\iota(z)=\sigma_1$ or $\sigma_{n-1}$, nor $\partial\phi(z) = \sigma_2$ or $\sigma_{n-2}$ have any strict prefixes!

This proves that the algorithm only gives the answers described in Theorem~\ref{T:ConjugRapide}.

Let us now study the complexity of this algorithm. According to~\cite{EM}, calculating the normal form $\fn(P_{45}P_{12})$ has computational complexity $O(\lc(P_{45}P_{12})^2) = O(\lc(x)^2)$. The tests (2) and (3) are carried out in constant time, and test (4) in linear time. Thus the total complexity of the algorithm is indeed $O(\lc(x)^2)$.

Finally, we have to prove that the proportion of braids for which the algorithm answers ``I don't know'' tends to zero exponentially fast as $l$~tends to infinity. This is a consequence of the properties shown in Section~\ref{bloquante}: the proportion of braids in the ball of radius~$l$ and center~$1$ in the Cayley graph satisfying the hypotheses of Observation~\ref{O:RigideCritere} (i.e.\ tests (2) and (3)) goes to $1$ exponentially quickly as $l$ goes to infinity. According to Lemma~\ref{lem:sousmot}, the same is true for the proportion of braids passing test~(4). In summary, the proportion of braids failing one of the tests (2), (3), or (4), and thus generating an answer ``I don't know'', tends to zero exponentially quickly.
\end{proof}

\begin{remark} 
In practice, test (4) should be replaced by ``test whether $P_3$ contains in its normal form a subword of the form $(\Delta \sigma_j^{-1}) \cdot \sigma_i$, $i \neq j$''. This would not change the algorithm's $O(\lc(x)^2)$ complexity, and it would further increase the proportion of braids for which the algorithm outputs a rigid conjugate, rather than answering ``I don't know''.
%tout en étendant l'ensemble des tresses pour lesquelles l'algorithme renvoie un conjugué rigide tel que l'ensemble des circuits glissants est réduit à l'orbite de la remarque \ref{remarque:orbiteSC}.
\end{remark}

%\section{Conclusions ?} 
\section{Further consequences and questions}\label{consequences}

\subsection{Balls containing only pseudo-Anosov braids}

\begin{corollary}
For every positive integer~$l$, there exists a vertex~$x$ in the Cayley graph of~$\B_n$ such that the ball of radius~$l$ centered in~$x$ contains only pseudo-Anosov elements.
\end{corollary}

\begin{proof}
Let us suppose, on the contrary, there is some number $l$ such that the whole Cayley graph is covered by $l$-balls around non pseudo-Anosov elements. This would mean that together, the $l$-balls centered on the non pseudo-Anosov elements in $\boule R$, the $R$-ball with center $1,$ cover the $(R-l)$-ball $\boule{R-l}$, for arbitrarily large $R$. (Notice that they would not necessarily cover the whole $R$-ball $\boule R$, because points that are $l$-close to its boundary might be covered by $l$-balls that are centered outside $\boule R$.)
%Si tout sommet du graphe de Cayley était à distance au plus~$l$ d'un élément non pseudo-anosovien, alors le graphe serait recouvert par des $l$-boules autour des éléments non pseudo-anosoviens. Ceci impliquerait que les $l$-boules autour des éléments non pseudo-anosoviens dans la $R$-boule $\boule R$ recouvraient la $(R-l)$-boule $\boule{R-l}$. (Attention, elles ne recouvriraient pas forcément toute la $R$-boule car les points $l$-proches du bord de celle-ci peuvent être recouvertes par des $l$-boules centrées en-dehors de la $R$-boule.) 
We deduce that
$$
\card{\beta \in \boule R, \beta \text{ non pseudo-Anosov}}\cdot \card{\boule l}\geqslant \card{\boule{R-l}}
$$
and therefore
$$
\frac{\card{\beta \in \boule R, \beta \text{ non pseudo-Anosov}}}{ \card{\boule{R}}}\geqslant \frac{1}{ \card{\boule{l}}}\cdot \frac{ \card{\boule{R-l}}}{ \card{\boule{R}}}.
$$
When $l$ is fixed and $R$ tends to infinity, the right hand side remains bounded below by a positive number, because the braid group is of exponential growth. This is in contradiction with Theorem~\ref{T:main}.
\end{proof}

We are grateful to Alessandro Sisto for pointing this corollary out to us. We have since learned from Saul Schleimer that this result was actually already known to certain specialists: it can also be proven by studying the action of~$\B_n$ on Thurston's compactification of Teichmüller space.

\subsection{The closure of a generic braid is a hyperbolic link}

\begin{theorem}
Consider the ball $\boule l$ of radius~$l$ and center $1$ in the Cayley graph of the braid group~$\B_n$, with generators the simple braids. Then, among the elements of this ball, the proportion of braids whose closure is a hyperbolic link tends to~1 as $l$ tends to infinity.
%Dans chaque groupe de tresses $\B_n$, parmi les tresses dans la $l$-boule du graphe de Cayley, la proportion de celles dont la fermeture est un entrelacs hyperbolique tend vers~$1$ quand $l$ tend vers l'infini.
\end{theorem}

\begin{proof}
A theorem of T.~Ito~\cite{Ito} states that a pseudo-Anosov braid $x$ which in Dehornoy's total order of the braid group \cite{DDRW} does not satisfy $\Delta^{-4}< x <\Delta^4$, has the property that its closure is a hyperbolic link. Thus by our main theorem~\ref{T:main}, it suffices to prove that, among the elements of~$\boule l$, the proportion of braids lying between $\Delta^{-4}$ and $\Delta^4$ in Dehornoy's order tends to~$0$ as $l$~tends to infinity. 

In order to do so, we recall that if a braid~$x$ satisfies $\Delta^{j-1}<x<\Delta^j$, then $\Delta x$ satisfies $\Delta^j<\Delta x<\Delta^{j+1}$. Now the $l$-ball in~$\B_n$ is the disjoint union 
$$
\boule{l} \ = \ \bigcup_{k=0}^l \ \bigcup_{x\in \B_n^{0,k}} \ \bigcup_{i=-l}^{l-k} \ \Delta^i x
$$
We conclude with the observation that, among the $2l-k+1$ elements $\Delta^i x$, with $-l\leqslant i\leqslant l-k$, there are at most five lying between $\Delta^{-4}$ and $\Delta^4$.
%On conclut par l'observation que, parmi les $2l-k+1$ éléments $\Delta^i x$ (pour $i$ entre $-l$ et $l-k$), il y a au plus cinq situés entre $-\Delta^{-4}$ et $\Delta^4$.
\end{proof}

\subsection{Questions}

It would be useful to extend our results to a much more general framework. From our proof, it is not even clear that Theorem~\ref{T:main} remains true if we replace Garside's generators with any other finite generating set, or if we replace $\B_n$ by a finite index subgroup (e.g.\ the pure braid group), or by its commutator subgroup, which is the kernel of the homomorphism $B_n\to \Z$ sending every Artin generator to~$1$.

For a start, one could try to adapt our arguments to the setting of general mapping class groups, equipped with Hamenstädt's bi-automatic structure~\cite{HamenstaedtBiAut}.

We conjecture that the analogue, for our notion of ``genericity'', of the main result of Sisto~\cite{sisto} holds. Specifically, let $G$ be a nonelementary group, equipped with a finite generating set and acting on a $\delta$-hyperbolic complex, where at least one element of~$G$ acts weakly properly discontinuously (WPD). Then we conjecture that the proportion of elements in the $l$-ball of the Cayley graph of~$G$ with a WPD action tends to~$1$ exponentially quickly as $l$~tends to infinity.
%Nous conjecturons que l'analogue dans notre cadre du résultat principal de Sisto~\cite{sisto} est vrai. Spécifiquement, soit $G$ un groupe non-élémentaire muni d'une famille génératrice finie qui agit sur un complexe $\delta$-hyperbolique, où au moins un élément de $G$ agit de façon faiblement proprement discontinue (WPD). Nous conjecturons que la proportion des éléments dans la $l$-boule du graphe de Cayley qui agissent de façon WPD tend vers~1 exponentiellement vite quand $l$~tend vers l'infini.
\bigskip

{\bf Acknowledgement } We thanks François Digne and Juan González-Menses for many helpful and detailed comments on earlier versions of this article.

\end{document}